\documentclass[11 pt]{article}
\usepackage{amssymb}
\usepackage{amsmath}
\usepackage{amsthm}
\allowdisplaybreaks
\usepackage{amscd}
\usepackage{graphicx}
\usepackage{hyperref}
\usepackage{geometry}
\usepackage{pxfonts}

\def\vbar{\mathchoice{\vrule height6.3ptdepth-.5ptwidth.8pt\kern- .8pt}
{\vrule height6.3ptdepth-.5ptwidth.8pt\kern-.8pt} {\vrule
height4.1ptdepth-.35ptwidth.6pt\kern-.6pt} {\vrule
height3.1ptdepth-.25ptwidth.5pt\kern-.5pt}}

\def\<{\langle}
\def\>{\rangle}

\def\c{\cdot}

\def\nw{\nwarrow}
\def\ne{\nearrow}

\newtheorem{thm}{Theorem}[section]
\newtheorem{lem}[thm]{Lemma}
\newtheorem{cor}[thm]{Corollary}
\newtheorem{pro}[thm]{Proposition}
\newtheorem{ex}[thm]{Example}

\theoremstyle{definition}
\newtheorem{defi}{Definition}[section]

\theoremstyle{remark}
\newtheorem{rmk}{Remark}[section]
\begin{document}
\title{  $3$-L-dendriform algebras and generalized derivations}
\author{\bf Taoufik Chtioui and Sami Mabrouk}
\author{Taoufik Chtioui$^{1}$
 \footnote{Corresponding author,  E-mail: chtioui.taoufik@yahoo.fr}
and Sami Mabrouk$^{2}$
 \footnote{Corresponding author,  E-mail: mabrouksami00@yahoo.fr}
\\
{\small 1.  Faculty of Sciences, University of Sfax, BP 1171, 3000 Sfax, Tunisia}\\
{\small 1.  University of Gafsa, Faculty of Sciences Gafsa, 2112 Gafsa, Tunisia}
}
\date{}
\maketitle
\begin{abstract}
The main goal of this paper is to introduce the notion of $3$-L-dendriform algebras  which are  the dendriform version of $3$-pre-Lie algebras. In fact they are the algebraic structures behind the $\mathcal{O}$-operator of $3$-pre-Lie algebras.
They can be also regarded as the ternary analogous of L-dendriform algebras.
Moreover, we study the generalized derivations of $3$-L-dendriform algebras. Finally, we  explore the spaces of quasi-derivations, the centroids and
the quasi-centroids  and give some properties.
\end{abstract}

\textbf{Key words}: 3-Lie algebras, 3-pre-Lie algebras, $3$-L-dendriform algebras, Representations, $\mathcal{O}$-operator, Generalized derivations.

\textbf{Mathematics Subject Classification}: 17A40,17A42,17B15.

\numberwithin{equation}{section}
\section*{Introduction}
A dendriform algebra is a module equipped with two binary products whose sum is associative. This concept was introduced by Loday in the late 1990s in the study of periodicity in algebraic $K$-theory \cite{L1}. Several years later, Loday and Ronco  introduced the concept of a tridendriform algebra from their study of algebraic topology \cite{AL1}. It is a module with three binary operations whose sum is associative. Afterward, quite a few similar algebraic structures were introduced, such as the quadi-algebra and ennea algebra \cite{AL}. The notion of splitting of associativity was introduced by Loday  to describe this phenomena in general for the associative operation (see \cite{L4, L5, L6}, for more details).

 A similar two-part and three-part splittings of the Lie-algebra are found to be the pre-Lie algebra and the post-Lie algebra repectively from operadic study  with applications to integrable systems \cite{burde, vallette}. Further, a two-part and three-part splittings of the associative commutative operation give the Zinbiel algebra  and commutative tridendriform algebra respectively \cite{L1, zinbiel}.
Analogues of the dendriform  and tridendriform algebra for  Jordan algebra, alternative algebra and Poisson algebra have also been obtained \cite{Bai3, prejordan, jdendriform}.

Some $n$-ary algebraic structures like  $n$-Lie algebras and $n$-associative algebras, which have been widely studied in the last few years,
were also decomposed into two and three operations (for $n=3$) giving
$3$-pre-Lie algebras and partially dendriform $3$-algebras \cite{PeiJun Chengming Bai and Li Guo}, using the
procedure that splits the operations in algebraic operads.

Rota-Baxter operator (of weight zero) was introduced independently in
the 1980s as the operator form of the classical Yang-Baxter equation ( see\cite{Baxter}), named after the physicists Yang
and Baxter. It  gives rise to splittings of various algebraic structures. This is the case, for example,  for associative algebras, giving the dendriform and tridendriform algebras, for Lie algebras, giving the pre-Lie and post-Lie algebras and for $3$-Lie algebras giving $3$-pre-Lie algebras.In \cite{Kupershmidt}, Kuperschmidt introduced the notion of Kuperschmidt  operators called also $\mathcal{O}$-operators which is generalisation of Rota-Baxter operators.

The following is an outline of the paper.  In section 1, we summarize  some definitions and known results  about $3$-Lie algebras  and $3$-pre-Lie algebras which will be useful in the sequel.  In section 2, we introduce the notion of $3$-L-dendriform algebra. We establish  that it has tow associated $3$-pre-Lie algebras (horizontal and vertical). They have the same sub-adjacent $3$-Lie algebra. In addition, the left multiplication  operator of  the first operation (called north-west: $\nw$) and the right multiplication operator of the second operation (called north-east: $\ne$) consist a bimodule of the associated horizontal $3$-pre-Lie algebra.
Section 3 is devoted to the study of generalized derivations on $3$-L-dendriform algebras, $3$-pre-Lie algebras and $3$-Lie algebras.

Throughout this paper $\mathbb{K}$ is a field of characteristic $0$ and all vector spaces are over $\mathbb{K}$.
\section{Preliminaries and basics}
\label{sec:bas}

In this section, we give some general results on $3$-Lie algebras  and  $3$-pre-Lie algebras which are useful throughout this paper.

\begin{defi}\cite{Filippov}
  A $3$-Lie algebra  consists  of a vector space $A$ equipped with a skew-symmetric linear map called $3$-Lie bracket $[\cdot,\cdot,\cdot]:
\otimes^3 A\rightarrow A$ such that the following Fundamental Identity holds (for $x_i\in A, 1\leq i\leq 5$)
\begin{equation}\label{eq:de1}
[x_1,x_2,[x_3,x_4,x_5]]=[[x_1,x_2,x_3],x_4,x_5]+[x_3,[x_1,x_2,x_4],x_5]+[x_3,x_4,[x_1,x_2,x_5]]
\end{equation}
\end{defi}
In other words, for $x_1, x_2\in A$, the operator
\begin{align}\label{eq:adjoint}
ad_{x_1,x_2}:A\to A, \quad ad_{x_1,x_2}x:=[x_1,x_2,x], \quad \forall x\in A,
\end{align}
is a derivation in the sense that
$$ ad_{x_1,x_2}[x_3,x_4,x_5]=[ad_{x_1,x_2}x_3,x_4,x_5] +[x_3,ad_{x_1,x_2}x_4,x_5]+[x_3,x_4,ad_{x_1,x_2}x_5], \forall x_3, x_4, x_5\in A.$$
A morphism between 3-Lie algebras is a linear map that preserves the 3-Lie brackets.
\begin{ex}\cite{Filippov} Consider $4$-dimensional $3$-Lie algebra $A$ generated by $(e_1,e_2,e_3,e_4)$  with  the following multiplication
\begin{align*}
          [e_1,e_2,e_3]= e_4, \
[e_1,e_2,e_4]= e_3,\
[e_1,e_3,e_4]= e_2, \
     [ e_2,e_3,e_4]= e_1.
\end{align*}
\end{ex}


The notion of a representation of an $n$-Lie algebra was introduced in \cite{repKasymov}. See also \cite{rep}.
\begin{defi}\label{defi:rep}
 A representation of a 3-Lie algebra $A$ on a vector space $V$ is a skew-symmetric linear map $\rho: \otimes^2A\rightarrow gl(V)$ satisfying
\begin{enumerate}
\item [\rm(i)] $\rho (x_1,x_2)\rho(x_3,x_4)-\rho(x_3,x_4)\rho(x_1,x_2)=
\rho([x_1,x_2,x_3],x_4)-\rho([x_1,x_2,x_4],x_3)$,
\item [\rm (ii)]$\rho ([x_1,x_2,x_3],x_4)=\rho(x_1,x_2)\rho(x_3,x_4)+\rho(x_2,x_3)\rho(x_1,x_4)+\rho(x_3,x_1)\rho(x_2,x_4)$,
\end{enumerate}
for $x_i\in A, 1\leq i\leq 4$.
\end{defi}

In fact,  $(V,\rho)$ is a representation of a 3-Lie algebra $(A,[\c,\c,\c])$ if and only if there is a $3$-Lie algebra structure
on the direct sum $A\oplus V$ of the underlying vector spaces $A$ and $V$ given by
\begin{equation}\label{eq:sum}
[x_1+v_1,x_2+v_2,x_3+v_3]_{A\oplus V}=[x_1,x_2,x_3]+\rho(x_1,x_2)v_3+\rho(x_3,x_1)v_2+\rho(x_2,x_3)v_1,
\end{equation}
for $x_i\in A, v_i\in V, 1\leq i\leq 3$. We denote it by $A\ltimes_\rho V.$

Now we recall
the definition of $3$-pre-Lie algebra and  exhibit construction results in terms of $\mathcal O$-operators  on $3$-Lie algebras (for more details see \cite{Bai1}, \cite{PeiJun Chengming Bai and Li Guo})

\begin{defi}\label{defi:o}
Let $(A,[\cdot,\cdot,\cdot])$ be a $3$-Lie algebra and $(V,\rho)$
a representation.  A linear operator $T:V\rightarrow A$ is called
an $\mathcal O$-operator associated to $( V,\rho)$ if $T$
satisfies
\begin{equation}\label{eq:Ooperator}
 [Tu,Tv,Tw]=T\left(\rho(Tu,Tv)w+\rho(Tv,Tw)u+\rho(Tw,Tu)v\right),\quad \forall u,v,w\in V.
\end{equation}
\end{defi}

\begin{defi}
A $3$-pre-Lie algebra is a pair $(A,\{\cdot,\cdot,\cdot\})$ consisting of a
a vector space $A$ and  a linear map $\{\cdot,\cdot,\cdot\}:A\otimes A\otimes A\rightarrow A$ such that the following identities hold:
\begin{eqnarray}
\{x,y,z\}&=&-\{y,x,z\},\label{3-pre-Lie 0}\\
\nonumber\{x_1,x_2,\{x_3,x_4,x_5\}\}&=&\{[x_1,x_2,x_3]^C,x_4,x_5\}+\{x_3,[x_1,x_2,x_4]^C,x_5\}\\
&&+\{x_3,x_4,\{x_1,x_2,x_5\}\},\label{3-pre-Lie 1}\\
\nonumber\{ [x_1,x_2,x_3]^C,x_4, x_5\}&=&\{x_1,x_2,\{ x_3,x_4, x_5\}\}+\{x_2,x_3,\{ x_1,x_4,x_5\}\}\\
&&+\{x_3,x_1,\{ x_2,x_4, x_5\}\},\label{3-pre-Lie 2}
\end{eqnarray}
 where $x,y,z, x_i\in A, 1\leq i\leq 5$ and $[\cdot,\cdot,\cdot]_C$ is defined by
\begin{equation}
[x,y,z]^C=\{x,y,z\}+\{y,z,x\}+\{z,x,y\},\quad \forall  x,y,z\in A.\label{eq:3cc}
\end{equation}
\end{defi}

\begin{ex}\label{Example1}
  Let $A$ be a $4$-dimensional vector space generated by $(e_1,e_2,e_3,e_4)$ and consider the bracket $\{\cdot,\cdot,\cdot\}:A\otimes A\otimes A\to A$ given by
$$\{e_1,e_2,e_3\}=-\{e_1,e_2,e_4\}=e_1+e_2$$
and all the other brackets are zero. Then  $(A,\{\cdot,\cdot,\cdot\})$ is a $3$-pre-Lie algebra.
\end{ex}

\begin{ex}\label{Example2}
Let $A$ be a $4$-dimensional vector space  generated by $(e_1,e_2,e_3,e_4)$ and let the bracket $\{\cdot,\cdot,\cdot\}:A\otimes A\otimes A\to A$ given by
$$\left\{
    \begin{array}{ll}
    \{e_1,e_2,e_3\}=e_4 ,  & \hbox{} \\
       \{e_1,e_2,e_4\}=e_3,   & \hbox{} \\
       \{e_1,e_3,e_3\}=-\{e_1,e_4,e_4\}=e_2  , & \hbox{} \\
    \{e_2,e_3,e_3\}=-\{e_2,e_4,e_4\}=e_1  ,  & \hbox{}
    \end{array}
  \right.
$$
 and all the other brackets are zero. Then  $(A,\{\cdot,\cdot,\cdot\})$ is a $3$-pre-Lie algebra.
\end{ex}
\begin{pro}
Let $(A,\{\cdot,\cdot,\cdot\})$ be a $3$-pre-Lie algebra. Then the induced $3$-commutator given by Eq.~\eqref{eq:3cc} defines
a $3$-Lie algebra.
\end{pro}

\begin{defi}
Let $(A,\{\cdot,\cdot,\cdot\})$ be a $3$-pre-Lie algebra. The $3$-Lie algebra $(A,[\cdot,\cdot,\cdot]^C)$
is called the  sub-adjacent $3$-Lie algebra of $(A,\{\cdot,\cdot,\cdot\})$ and $(A,\{\cdot,\cdot,\cdot\})$ is called a compatible
$3$-pre-Lie algebra of the $3$-Lie algebra $(A,[\cdot,\cdot,\cdot]^C)$.
\end{defi}

Let $(A,\{\cdot,\cdot,\cdot\})$ be a $3$-pre-Lie algebra. Define a skew-symmetric  linear map $L: \otimes^2A\rightarrow  gl(A)$
by
\begin{equation}\label{eq:R}
L(x,y)z=\{x,y,z\},\quad \forall x,y,z\in A.
\end{equation}

By the definitions of a $3$-pre-Lie algebra and a representation of a $3$-Lie algebra, we immediately obtain

\begin{pro}
With the above notations, $(A,L)$ is a representation of the
$3$-Lie algebra $(A,[\cdot,\cdot,\cdot]^C)$. On the other hand,
let $A$ be a vector space with a linear map
$\{\cdot,\cdot,\cdot\}:A\otimes A\otimes A\rightarrow A$
satisfying
 Eq.~\eqref{3-pre-Lie 0}. Then $(A,\{\cdot,\cdot,\cdot\}) $ is a $3$-pre-Lie algebra if $[\cdot,\cdot,\cdot]^C$ defined by Eq~\eqref{eq:3cc} is a $3$-Lie algebra and the left multiplication $L$ defined by Eq.~\eqref{eq:R}
gives a representation of this $3$-Lie algebra.
\end{pro}

\begin{pro}\label{pro:3preLieT}
Let $(A,[\cdot,\cdot,\cdot])$ be a $3$-Lie algebra and $(V,\rho)$ a representation. Suppose that the linear map $T:V\rightarrow A$ is an $\mathcal O$-operator associated
to $(V,\rho)$. Then there exists a $3$-pre-Lie algebra structure on $V$ given by
\begin{equation}
\{u,v,w\}=\rho(Tu,Tv)w,\quad\forall ~ u,v,w\in V.
\end{equation}
\end{pro}

\begin{cor}
With the above conditions,  $(V,[\cdot,\cdot,\cdot]_C)$ is a $3$-Lie
algebra as the sub-adjacent $3$-Lie algebra of the $3$-pre-Lie
algebra given in Proposition \ref{pro:3preLieT}, and $T$ is a $3$-Lie algebra morphism from $(V,[\cdot,\cdot,\cdot]_C)$ to $(A,[\cdot,\cdot,\cdot])$. Furthermore,
$T(V)=\{Tv|v\in V\}\subset A$ is a $3$-Lie subalgebra of $A$ and there is an induced $3$-pre-Lie algebra structure $\{\cdot,\cdot,\cdot\}_{T(V)}$ on
$T(V)$ given by
\begin{equation}
\{Tu,Tv,Tw\}_{T(V)}:=T\{u,v,w\},\quad\;\forall u,v,w\in V.
\end{equation}
\end{cor}

\begin{pro}\label{pro:preLieOoper}
Let $(A,[\cdot,\cdot,\cdot])$ be a $3$-Lie algebra. Then there exists a compatible $3$-pre-Lie algebra if and only if there exists an invertible $\mathcal O$-operator on $A$.
\end{pro}

\begin{defi}\cite{Sheng Yunhe and Rong Tang}
 A  representation of a $3$-pre-Lie algebra $(A,\{\cdot,\cdot,\cdot\})$   on a vector space $V$ consists of a pair $(l,r)$, where $l:\wedge^2 A \rightarrow gl(V)$ is a representation of the $3$-Lie algebra $A^c$ on $V$ and $r:\otimes^2 A \rightarrow gl(V)$ is a linear map such that  for all $x_1,x_2,x_3,x_4\in A$, the following equalities hold:
\begin{eqnarray}
 l(x_1,x_2)r(x_3,x_4) &=&r(x_3,x_4)\mu(x_1,x_2)
                          \label{rep1}+r([x_1,x_2,x_3]_C,x_4)+ r(x_3,\{x_1,x_2,x_4\}),\\
\label{rep2} r([x_1,x_2,x_3]_C,x_4)&=&l(x_1,x_2)r(x_3,x_4)+l(x_2,x_3)r(x_1,x_4)+
l(x_3,x_1)r(x_2,x_4),\\
 r(x_1,\{x_2,x_3,x_4\}) &=&r(x_3,x_4)\mu(x_1,x_2)-r(x_2,x_4)\mu(x_1,x_3)
                         + l(x_2,x_3)r(x_1,x_4),  \label{rep3}\\
r(x_3,x_4)\mu(x_1,x_2) &=&
                       l(x_1,x_2)r(x_3,x_4)-r(x_2,\{x_1,x_3,x_4\})+
                        r(x_1,\{x_2,x_3,x_4\}),  \label{rep4}
\end{eqnarray}
\end{defi}
where $\mu(x,y)=l(x,y)+r(x,y)-r(y,x)$, for any $x,y \in A$.

Define the left multiplication $L:\wedge^2 A\longrightarrow gl(A)$ by
$L(x,y)z=\{x,y,z\}$ for all $x,y,z\in A$. Then $(A,L)$ is a representation of the
$3$-Lie algebra $A^c$. Moreover, we define the right multiplication
$R:\otimes^2 A \to  gl(A)$ by $R(x,y)z=\{z,x,y\}$.
It is obvious that $(A,L,R)$ is a
representation of a $3$-pre-Lie algebra on itself, which is called the adjoint
representation.

\begin{pro}
Let $(A,\{\c,\c,\c\})$ be a $3$-pre-Lie algebra, $V$  a vector space and $l,r:
\otimes^2A\rightarrow  gl(V)$  two  linear
maps. Then $(V,l,r)$ is a representation of $A$ if and only if there
is a $3$-pre-Lie algebra structure $($called the semi-direct product$)$
on the direct sum $A\oplus V$ of vector spaces, defined by
\begin{equation}\label{eq:sum}
[x_1+u_1,x_2+u_2,x_3+u_3]_{A\oplus V}=\{x_1,x_2,x_3\}+l(x_1,x_2)u_3-r(x_1,x_3)u_2+r(x_2,x_3)u_1,
\end{equation}
for $x_i\in A, u_i\in V, 1\leq i\leq 3$. We denote this semi-direct product $3$-Lie algebra by $A\ltimes_{l,r} V.$
\end{pro}

Let $V$ be a vector space.   Define the switching operator
$\tau:\otimes^2 V\longrightarrow \otimes^2 V$ by
\begin{eqnarray*}
\tau(T)=x_2\otimes x_1,\quad \forall T=x_1\otimes x_2\in\otimes^2 V.
\end{eqnarray*}

  \begin{pro}
Let $(V,l,r)$ be a representation of a $3$-pre-Lie algebra $(A,\{\cdot,\cdot,\cdot\})$. Then $l-r\tau+r$ is a representation of the
sub-adjacent $3$-Lie algebra $(A^c,[\cdot,\cdot,\cdot]_C)$ on the vector space $V$.
  \end{pro}

\begin{pro}
Let $(l,r)$ be a representation  of a $3$-pre-Lie algebra $(A,\{\cdot,\cdot,\cdot\})$ on a vector space $V$.  Then $(l^*-r^* \tau +r^*,-r^*)$ is a representation of the  $3$-pre-Lie algebra $(A,\{\cdot,\cdot,\cdot\})$  on the vector space $V^*$, which is called the dual representation of the representation $(V,l,r)$.
\end{pro}

 If $(l,r)=(L,R)$ is the adjoint representation of a $3$-pre-Lie algebra $(A,\{\cdot,\cdot,\cdot\})$, then we obtain
$(l^*-r^*\tau+r^*,-r^*)=(ad^*,-R^*)$.

\begin{defi}
 Let $(A,\{\c,\c,\c\})$ be a $3$-pre-Lie algebra and $(V,l,r)$ be a representation. A linear operator $T: V \to A$ is called and $\mathcal{O}$-operator associated to $(V,l,r)$ if $T$ satisfies
 \begin{align}\label{O-op 3-pre-Lie}
    \{Tu,Tv,Tw\}=T\left(l(Tu,Tv)w-r(Tu,Tw)v+r(Tv,Tw)u\right),\quad \forall u,v,w\in V.
 \end{align}
\end{defi}
If $V=A$, then $T$ is called a Rota-Baxter operator on $A$ of weight zero. That is
\begin{align*}
[R(x),R(y),R(z)]= R\big([R(x),R(y),z]+[R(x),y,R(z)]+[x,R(y),R(z)]\big),
\end{align*}
for all $x,y,z \in A$.

\begin{ex}
Let the $4$-dimensional $3$-pre-Lie algebra given in Example \ref{Example1}. Define $R:A\to A$ by
$$R(e_1) = e_1 + e_2,\ R(e_2) = e_3 + e_4,\ R(e_3) = R(e_4) = 0.$$
By a direct computation, we can verify that $R$ is a Rota-Baxter operator.
\end{ex}


\section{$3$-L-dendriform algebras}
In this section, we introduce the notion of a $3$-L-dendriform algebra which is exactly the ternary version of a L-dendriform algebra. We provide some construction results in terms of $\mathcal{O}$-operator and symplectic  structure.

\begin{defi}
Let $A$ be a vector space with two linear maps $\nwarrow, \nearrow : \otimes^3 A \to A$. The tuple $(A,\nwarrow,\nearrow)$ is called a $3$-L-dendriform algebra if the following identities hold
\begin{align}\label{3-L-dendriform0}
&  \nw(x_1,x_2,x_3)+ \nw(x_2,x_1,x_3)=0  ,\\
 \label{3-L-dendriform1}&   \nw(x_1,x_2,\nw(x_3,x_4,x_5))-  \nw(x_3,x_4,\nw(x_1,x_2,x_5)) \nonumber\\
&\hspace{2 cm} =\nw([x_1,x_2,x_3]^C,x_4,x_5)-\nw([x_1,x_2,x_4]^C,x_3,x_5)      ,\\
 &\label{3-L-dendriform2}  \nw(x_1,x_2,\ne(x_5,x_3,x_4))-\ne(x_5,x_3,\{x_1,x_2,x_4\}^h) \nonumber\\
& \hspace{2 cm} =\ne(x_5,[x_1,x_2,x_3]^C,x_4)+\ne(\{x_1,x_2,x_5\}^v,x_3,x_4)      , \\
 & \label{3-L-dendriform3}\ne(x_5,x_1,\{x_2,x_3,x_4\}^h)-\nw(x_2,x_3,\ne(x_5,x_1,x_4))
 \nonumber\\
&\hspace{2 cm}=\ne(\{x_1,x_2,x_5\}^v,x_3,x_4)-\ne(\{x_1,x_3,x_5\}^v,x_2,x_4)
      ,\\
 &  \label{3-L-dendriform4}\nw([x_1,x_2,x_3]^C,x_4,x_5)=\circlearrowleft_{1,2,3}\nw(x_1,x_2,\nw(x_3,x_4,x_5)),
   \\
 &   \label{3-L-dendriform5}     \ne(x_5,[x_1,x_2,x_3]^C,x_4)=\circlearrowleft_{1,2,3}\nw(x_1,x_2,\ne(x_5,x_3,x_4)) ,\\
 &   \label{3-L-dendriform6}   \nw(x_1,x_2,\ne(x_5,x_3,x_4))+\ne(x_5,x_1,\{x_2,x_3,x_4\}^h)\nonumber\\
&\hspace{2 cm} =\ne(x_5,x_2,\{x_1,x_3,x_4\}^h)+\ne(\{x_1,x_2,x_5\}^v,x_3,x_4)     ,
\end{align}
 for all $x_i \in A$, $1\leq i \leq 5$, where
\begin{align}
& \{x,y,z\}^h=\nw(x,y,z)+\ne(x,y,z)-\ne(y,x,z)  ,   \label{accolade horizintal}\\
& \{x,y,z\}^v=\nw(x,y,z)+\ne(z,x,y)-\ne(z,y,x)  ,   \label{accolade vertical}\\
& [x,y,z]^C= \circlearrowleft_{x,y,z} \{x,y,z\}^h=\circlearrowleft_{x,y,z} \{x,y,z\}^v, \label{crochet}
\end{align}
for any $x,y,z \in A$.
\end{defi}

\begin{rmk}
Let $(A,\nw,\ne)$ be a $3$-L-dendriform algebra.  if $\ne=0$, then $(A,\nw)$ is a $3$-pre-Lie algebra.
\end{rmk}

\begin{pro}\label{3LDendTo3PreLie}
Let $(A,\nw,\ne)$ be a $3$-L-dendriform algebra.
\begin{enumerate}
  \item  The bracket given in \eqref{accolade horizintal} defines a $3$-pre-Lie algebra structure on $A$ which is called the associated  horizontal  $3$-pre-Lie algebra of  $(A,\nw,\ne)$  and $(A,\nw,\ne)$ is also called  a compatible $3$-L-dendriform algebra structure  on the  $3$-pre-Lie algebra $(A,\{\c,\c,\c\}^h)$.
  \item  The bracket given in \eqref{accolade vertical}  defines a $3$-pre-Lie algebra structure on $A$ which is called the associated vertical  $3$-pre-Lie algebra of  $(A,\nw,\ne)$  and $(A,\nw,\ne)$ is also called  a compatible $3$-L-dendriform algebra structure  on the  $3$-pre-Lie algebra $(A,\{\c,\c,\c\}^v)$.
\end{enumerate}

\end{pro}

\begin{proof}
We will just prove item 1.
Note, first that $\{x,y,z\}^h=-\{y,x,z\}^h$ and $\{x,y,z\}^v=-\{y,x,z\}^v$ , for any $x,y,z \in A$.  \\
Let $x_i \in A,\ 1 \leq i \leq 5$. Then
\begin{align*}
&\{x_1,x_2,\{x_3,x_4,x_5\}^h\}^h-\{x_3,x_4,\{x_1,x_2,x_5\}^h\}^h -\{[x_1,x_2,x_3]^C,x_4,x_5\}^h+\{[x_1,x_2,x_4]^C,x_3,x_5\}^h\\
&= r_1+ r_2+r_3+r_4+r_5,
\end{align*}
where
\begin{align*}
r_1&= \nw(x_1,x_2,\nw(x_3,x_4,x_5))-  \nw(x_3,x_4,\nw(x_1,x_2,x_5))  -\nw([x_1,x_2,x_3]^C,x_4,x_5)\\
&+\nw([x_1,x_2,x_4]^C,x_3,x_5), \\
r_2&= \ne(x_3,[x_1,x_2,x_4]^C,x_5)+\ne(x_3,x_4,\{x_1,x_2,x_5\}^h) - \nw(x_1,x_2,\ne(x_3,x_4,x_5))  \\
&+\ne(\{x_1,x_2,x_3\}^v,x_4,x_5) ,\\
r_3&= \ne(x_4,[x_1,x_2,x_3]^C,x_5)+\ne(x_4,x_3,\{x_1,x_2,x_5\}^h)- \nw(x_1,x_2,\ne(x_4,x_3,x_5))\\
&  +\ne(\{x_1,x_2,x_4\}^v,x_3,x_5),\\
r_4&= \ne(x_1,x_2,\{x_3,x_4,x_5\}^h)-\nw(x_3,x_4,\ne(x_1,x_2,x_5))
 \nonumber-\ne(\{x_2,x_3,x_1\}^v,x_4,x_5)
 \\
&+\ne(\{x_2,x_4,x_1\}^v,x_3,x_5) ,\\
 r_5&= \ne(x_2,x_1,\{x_3,x_4,x_5\}^h)-\nw(x_3,x_4,\ne(x_2,x_1,x_5))
 \nonumber-\ne(\{x_1,x_3,x_2\}^v,x_4,x_5)
\\
& +\ne(\{x_1,x_4,x_2\}^v,x_3,x_5).
\end{align*}
From identities \eqref{3-L-dendriform1}-\eqref{3-L-dendriform3}, we obtain immediately $r_i=0,\ \forall 1 \leq i \leq 5$.
This imply that \eqref{3-pre-Lie 1} hold.

On the other hand, we have
\begin{align*}
&\{ [x_1,x_2,x_3]^C,x_4, x_5\}^h-\{x_1,x_2,\{ x_3,x_4, x_5\}^h\}^h-\{x_2,x_3,\{ x_1,x_4,x_5\}^h\}^h -\{x_3,x_1,\{ x_2,x_4, x_5\}^h\}^h \\
=& s_1+s_2+s_3+s_4+s_5,
\end{align*}
where
\begin{align*}
&s_1=\nw([x_1,x_2,x_3]^C,x_4,x_5)-
\circlearrowleft_{1,2,3}\nw(x_1,x_2,\nw(x_3,x_4,x_5)),\\
&s_2= \ne(x_4,[x_1,x_2,x_3]^C,x_5)-
\circlearrowleft_{1,2,3}\nw(x_1,x_2,\ne(x_4,x_3,x_5)) ,\\
&s_3=\nw(x_1,x_2,\ne(x_3,x_4,x_5))+\ne(x_3,x_1,\{x_2,x_3,x_5\}^h)\\
&\hspace{2 cm} -\ne(x_3,x_2,\{x_1,x_4,x_5\}^h)-\ne(\{x_1,x_2,x_3\}^v,x_4,x_5)      ,\\
&s_4= \nw(x_2,x_3,\ne(x_1,x_4,x_5))+\ne(x_1,x_2,\{x_3,x_4,x_5\}^h)\\
&\hspace{2 cm} -\ne(x_1,x_3,\{x_2,x_4,x_5\}^h)-\ne(\{x_2,x_3,x_1\}^v,x_4,x_5)   ,\\
&s_5=  \nw(x_3,x_1,\ne(x_2,x_4,x_5))+\ne(x_2,x_3,\{x_1,x_4,x_5\}^h)\\
&\hspace{2 cm} -\ne(x_2,x_1,\{x_3,x_4,x_5\}^h)-\ne(\{x_3,x_1,x_2\}^v,x_4,x_5)  .\\
\end{align*}
From identities \eqref{3-L-dendriform4}-\eqref{3-L-dendriform6}, we obtain immediately $s_i=0,\ \forall 1 \leq i \leq 5$.
This imply that \eqref{3-pre-Lie 2} hold.

\end{proof}

\begin{cor}
Let $(A,\nw,\ne)$ be a $3$-L-dendriform algebra. Then the bracket defined in  \eqref{crochet} defines a $3$-Lie algebra structure on $A$ which is called the associated   $3$-Lie algebra of  $(A,\nw,\ne)$.
\end{cor}
The following Proposition is obvious.
\begin{pro}
Let $(A,\nw,\ne)$ be a $3$-L-dendriform algebra. Define $L_{\nw},R_{\ne}: \otimes^2 A \to gl(A)$ by
$$L_{\nw}(x,y)z=\nw(x,y,z),\quad  R_{\ne}(x,y)z=\ne(z,x,y), \ \rho(x,y)z=\nw(x,y,z)+\ne(z,x,y)-\ne(z,y,x)$$
 for all $x,y,z \in A.$
Then
\begin{enumerate}
\item[(1)] $(A,L_{\nw},R_{\ne})$ is a representation of  its  horizontal associated   $3$-pre-Lie algebra $(A,\{\c,\c,\c\}^h)$.
\item[(2)]  $(A,L_{\nw})$ is a representation of its associated $3$-Lie algebra $(A,[\c,\c,\c]^C)$.
\item[(3)] $(A,\rho)$ is a representation of its associated $3$-Lie algebra $(A,[\c,\c,\c]^C)$.
\end{enumerate}
\end{pro}

\begin{rmk}
In the sense of the above conclusion (1), a $3$-L-dendriform algebra is understood as
a ternary  algebra structure whose left and right multiplications give a bimodule structure on the underlying vector space of the $3$-pre-Lie algebra defined by certain commutators. It can be regarded as the \textbf{rule} of introducing the notion of $3$-L-dendriform algebra, which more generally, is the \textbf{rule}
of introducing the notions of $3$-pre-Lie algebras,  the Loday algebras and their Lie,Jordan and alternative  algebraic analogues.
\end{rmk}

\begin{thm}
 Let $(A,\{\c,\c,\c\})$ be a $3$-pre-Lie algebra and $(V,l,r)$ be a representation. Suppose that  $T: V \to A$ is an $\mathcal{O}$-operator associated to $(V,l,r)$. Then there exists a $3$-L-dendriform algebra structure on $V$ given by
 \begin{align}
  \nw(u,v,w)=l(Tu,Tv)w  , \quad \ne(u, v,w)= r(Tv,Tw)u, \forall u,v,w \in V.
 \end{align}
 Therefore, there exists two associated $3$-pre-Lie algebra structures on $V$ and $T$ is a homomorphism of $3$-pre-Lie algebras. Moreover, $T(V)=\{T(v)| v \in V \}$ is $3$-pre-Lie subalgebra of $(A,\{\c,\c,\c\})$ and there is an induced $3$-L-dendriform algebra structure on $T(V)$ given by
  \begin{align}
  \nw(Tu,Tv,Tw)=T(  \nw(u,v,w))  , \quad \ne(Tu, Tv,Tw)= T(\ne(u, v,w)), \forall u,v,w \in V.
 \end{align}

\end{thm}

\begin{proof}
Let $u,v,w \in V$.  Define $\{\c,\c,\c\}_V^h, \{\c,\c,\c\}_V^v,[\c,\c,\c]^C_V:\otimes^3V \to V$ by
\begin{align*}
& \{u,v,w\}_V^h=\nw(u,v,w)+\ne(u,v,w)-\ne(v,u,w),\\
& \{u,v,w\}_V^v=\nw(u,v,w)+\ne(w,u,v)-\ne(w,v,u),\\
&[u,v,w]^C_V=\circlearrowleft_{u,v,w}\{u,v,w\}_V^h.
\end{align*}
Using identity \eqref{O-op 3-pre-Lie}, we have
\begin{align*}
T\{u,v,w\}_V^h &=T(\nw(u,v,w)+\ne(u,v,w)-\ne(v,u,w))\\
&=T(l(Tu,Tv)w-r(Tu,Tw)v+r(Tv,Tw)u)=\{Tu,Tv,Tw\}^h
\end{align*}
and
$$T[u,v,w]^C_V=\circlearrowleft_{u,v,w}T\{u,v,w\}_V
=\circlearrowleft_{u,v,w}\{Tu,Tv,Tw\}^h=[Tu,Tv,Tw]^C.$$
It is straightforward that
$$\nw(u,v,w)+ \nw(v,u,w)= ( l(Tu,Tv)+l(Tv,Tu))w=0.$$
Furthermore, for any $u_i \in V,\ 1 \leq i \leq 5$, we have
\begin{align*}
&   \nw(u_1,u_2,\nw(u_3,u_4,u_5))-  \nw(u_3,u_4,\nw(u_1,u_2,u_5)) \\
&-\nw([u_1,u_2,u_3]_V^C,u_4,u_5)+\nw([u_1,u_2,u_4]_V^C,u_3,u_5) \\
&= l(T(u_1),T(u_2))l(T(u_3),T(u_4))u_5-  l(T(u_3),T(u_4))l(T(u_1),T(u_2))u_5 \\
&-l(T[u_1,u_2,u_3]_V^C,T(u_4))u_5+l(T[u_1,u_2,u_4]_V^C,T(u_3))u_5 \\
&=\Big( [l(T(u_1),T(u_2)),l(T(u_3),T(u_4))]-l([T(u_1),T(u_2),T(u_3)]^C,T(u_4)) \\
&\hspace{1 cm}+ l([T(u_1),T(u_2),T(u_4)]^C,T(u_3)) \Big)u_5
=0.
\end{align*}
This implies that \eqref{3-L-dendriform1} holds.
Moreover, \eqref{3-L-dendriform2} holds. Indeed,
\begin{align*}
& \ne(u_5,[u_1,u_2,u_3]_V^C,u_4)-\ne(u_5,u_3,\{u_1,u_2,u_4\}^h_V) \\
&- \nw(u_1,u_2,\ne(u_5,u_3,u_4)) +\ne(\{u_1,u_2,u_5\}^v_V,u_3,u_4) \\
&=r(T[u_1,u_2,u_3]_V^C,T(u_4))u_5-r(T(u_3),T\{u_1,u_2,u_4\}^h_V)u_5 \\
&-l(T(u_1),T(u_2))r(T(u_3),T(u_4))u_5 +r(T(u_3),T(u_4))\{u_1,u_2,u_5\}^v_V \\
&=r([Tu_1,Tu_2,Tu_3]^C,T(u_4))u_5-r(T(u_3),\{T(u_1),T(u_2),T(u_4)\}^h)u_5 \\
&-l(T(u_1),T(u_2))r(T(u_3),T(u_4))u_5 +r(T(u_3),T(u_4))\mu(T(u_1),T(u_2))u_5 \\
=0
\end{align*}
To prove identity \eqref{3-L-dendriform6}, we compute as follow
\begin{align*}
& \nw(u_1,u_2,\ne(u_5,u_3,u_4))+\ne(u_5,u_1,\{u_2,u_3,u_4\}^h_V) \\
&\hspace{1 cm} -\ne(u_5,u_2,\{u_1,u_3,u_4\}^h_V)-\ne(\{u_1,u_2,u_5\}^v_V,u_3,u_4)  \\
    &=l(Tu_1,Tu_2)r(Tu_3,Tu_4)u_5+r(Tu_1,\{Tu_2,Tu_3,Tu_4\}^h)u_5  \\
    & \hspace{1 cm} -r(Tu_2,\{Tu_1,Tu_3,Tu_4\}^h)u_5-r(Tu_3,Tu_4)\mu(Tu_1,Tu_2)u_5 \\
    &=0.
\end{align*}
The other conclusions follow immediately.
\end{proof}
\begin{cor}\label{CorLDendViaRB}
   Let $(A,\{\c,\c,\c\})$ be a $3$-pre-Lie algebra and  $R: A \to A$ is a Rota-Baxter operator of weight $0$. Then there exists a $3$-L-dendriform algebra structure on $A$ given by
 \begin{align}\label{LDendViaRB}
  \nw(x,y,z)=\{R(x),R(y),z\}  , \quad \ne(x, y,z)= \{x,R(y),R(z)\},  \end{align}
for all $x,y,z \in A$.

\end{cor}

\begin{ex}
Let the $4$-dimensional $3$-pre-Lie algebra given in Example \ref{Example2}. Define $R:A\to A$ by
$$R(e_1) =  e_2,\ R(e_2) = -e_1,\ R(e_3) =e_4,\ R(e_4) = e_3.$$
Then  $R$ is a Rota-Baxter operator. Using Corollary \ref{CorLDendViaRB}, we can construct a $3$-Lie dendriform algebra given by the structures $\nw,\ne:A\otimes A\otimes A\to A$ defined in  the basis $(e_1,e_2,e_3,e_4)$, by
$$\left\{
    \begin{array}{ll}
      \nw(e_1,e_2,e_3)=e_4, & \hbox{} \\
\nw(e_1,e_2,e_4)=e_3, & \hbox{} \\
\ne(e_1,e_1,e_3)=\ne(e_2,e_2,e_3)=e_3, & \hbox{} \\
\ne(e_1,e_1,e_4)=\ne(e_2,e_2,e_4)=e_4, & \hbox{} \\
\ne(e_1,e_3,e_3)=-\ne(e_1,e_4,e_4)=-e_2, & \hbox{} \\
      \ne(e_2,e_3,e_3)=-\ne(e_2,e_4,e_4)=-e_1 & \hbox{}
    \end{array}
  \right.
$$
and all the other products are zero.
  \end{ex}

\begin{pro}\label{3Lden by invert O-op}
 Let $(A,\{\c,\c,\c\})$ be a $3$-pre-Lie algebra.  Then there exists a compatible
3-L-dendriform  algebra if and only if there exists an invertible $\mathcal{O}$-operator on $A$.
\end{pro}
 \begin{proof}
 Let $T$ be an invertible $\mathcal{O}$-operator of $A$ associated to a representation $(V, l,r)$.
Then there exists a $3$-L-dendriform algebra structure on $V$ defined by
 \begin{align}
  \nw(u,v,w)=l(Tu,Tv)w  , \quad \ne(u, v,w)= r(Tv,Tw)u, \forall u,v,w \in V.
 \end{align}
In addition there exists a $3$-L-dendriform algebra structure on $T(V)=A$ given by
 \begin{align}
  \nw(Tu,Tv,Tw)=T(l(Tu,Tv)w)  , \quad \ne(Tu, Tv,Tw)= T(r(Tv,Tw)u), \forall u,v,w \in V.
 \end{align}
 If we put $x=Tu,\ y=Tv$ and $z=Tw$, we get
  \begin{align}
  \nw(x,y,z)=T(l(x,y)T^{-1}(z))  , \quad \ne(x, y,z)= T(r(y,z)T^{-1}(x)), \forall w,y,z \in A.
 \end{align}
It is a compatible $3$-L-dendriform algebra structure on $A$. Indeed,
\begin{align*}
&  \nw(x,y,z)+\ne(x,y,z)-\ne(y,x,z) \\
=& T(l(x,y)T^{-1}(z)) + T(r(y,z)T^{-1}(x))-T(r(x,z)T^{-1}(y)) \\
=& \{TT^{-1}(x),TT^{-1}(y),TT^{-1}(z)\}=\{x,y,z\}.
\end{align*}
Conversely,   let $(A,\{\c,\c,\c\})$ be a $3$-pre-Lie algebra and $(A,\nw,\ne)$ its compatible $3$-L-dendriform algebra.  Then the identity map $id: A \to A$ is an $\mathcal{O}$-operator of  $(A,\{\c,\c,\c\})$ associated to $(A,L,R)$.
 \end{proof}

\begin{defi}
Let $(A,\{\c,\c,\c\})$ be a $3$-pre-Lie algebra and $B$ be a skew-symmetric bilinear form on $A$.  We say that $B$ is closed if it satisfies
\begin{align}\label{symplectic bilinear form}
B(\{x,y,z\},w)-B(z,[x,y,w]^C)-B(y,\{w,x,z\})+B(x,\{w,y,z\})=0,
\end{align}
for any $x,y,z,w \in A$.  If in addition $B$ is nondegenerate  , then $B$ is called symplectic.

A $3$-pre-Lie algebra $(A,\{\c,\c,\c\})$ equipped with a symplectic form is called a symplectic $3$-pre-Lie algebra and denoted by  $(A,\{\c,\c,\c\},B)$.
\end{defi}

\begin{pro}
Let $(A,\{\c,\c,\c\},B)$ be a symplectic $3$-pre-Lie algebra. Then there exists a compatible $3$-L-dendriform algebra structure on $A$ given by
\begin{align}\label{3Ldend by  form}
B(\nw(x,y,z),w)=B(z,[x,y,w]^C),\quad B(\ne(x,y,z),w)=-B(x,\{w,y,z\}), \forall x,y,z,w \in A.
\end{align}
\end{pro}

\begin{proof}
Define the linear map $T: A^* \to A$ by $\langle T^{-1}x,y\rangle=B(x,y)$.  Using Eq. \eqref{symplectic bilinear form}, we obtain that $T$ is an invertible $\mathcal{O}$-operator  on $A$ associated to the dual representation $(A^*, ad^*,-R^*)$.   By Proposition \ref{3Lden by invert O-op}, there exists a compatible $3$-L-dendriform algebra structure given by
  \begin{align}
  \nw(x,y,z)=T(ad^*(x,y)T^{-1}(z))  , \quad \ne(x, y,z)= -T(R^*(y,z)T^{-1}(x)), \forall w,y,z \in A.
 \end{align}
Then we have
\begin{align*}
    B(\nw(x,y,z),w)= & B(T(ad^*(x,y)T^{-1}(z)) ,w)=\langle ad^*(x,y)T^{-1}(z),w\rangle \\
    =& \langle T^{-1}(z),[x,y,w]^C\rangle=B(z,[x,y,w]^C)
\end{align*}
and
\begin{align*}
    B(\ne(x,y,z),w)= & -B(T(R^*(y,z)T^{-1}(x))) ,w)=-\langle R^*(y,z)T^{-1}(x),w\rangle \\
    = &-\langle T^{-1}(x),\{w,y,z\}\rangle=-B(x,\{w,y,z\}).
\end{align*}
The proof is finished.
\end{proof}

\begin{cor}
Let $(A,\{\c,\c,\c\},B)$ be a symplectic $3$-pre-Lie algebra and let $(A,[\c,\c,\c]^C)$ be its associated $3$-Lie algebra.  Then there exists a $3$-pre-Lie algebraic structure $(A,\{\c,\c,\c\}')$ on $A$ given by
\begin{align}
B(\{x,y,z\}',w)=B(z,[x,y,w]^C)-B(z,\{w,x,y\})+B(z,\{w,y,x\})=0.
\end{align}
\end{cor}

\begin{lem}\label{commuting rota-baxter op}
Let $\{R_1,R_2\}$ be a pair of of commuting Rota-Baxter operators (of weight zero) on a $3$-Lie algebra $(A, [\c,\c,\c])$. Then $R_2$ is a Rota-Baxter operator (of weight zero) on the associated $3$-pre-Lie algebra
defined by  $\{x,y,z\}=[R_1(x),R_1(y),z]$.
\end{lem}

\begin{proof}
For any $x,y,z \in A$, we have
\begin{align*}
&\{R_2(x),R_2(y),R_2(z)\}=[R_1R_2(x),R_1R_2(y),R_2(z)] =[R_2R_1(x),R_2R_1(y),R_2(z)] \\
&=R_2([R_2R_1(x),R_2R_1(y),z]+[R_1(x),R_2R_1(y),R_2(z)]  +
[R_2R_1(x),R_1(y),R_2(z)] ) \\
&=R_2(\{R_2(x),R_2(y),z\}+\{x,R_2(y),R_2(z)\}+\{R_2(x),y,R_2(z)\}).
\end{align*}
Hence $R_2$ is a Rota-Baxter operator (of weight zero) on the $3$-pre-Lie algebra
$(A,\{\c,\c,\c\})$.
\end{proof}

\begin{pro}
Let $\{R_1,R_2\}$ be a pair of of commuting Rota-Baxter operators (of weight zero) on a $3$-Lie algebra $(A, [\c,\c,\c])$. Then there exists a $3$-L-dendriform algebra structure on $A$ defined by
\begin{align*}
\nw(x,y,z)=[R_1R_2(x),R_1R_2(y),z],\quad \ne(x,y,z)=[R_1(x),R_1R_2(y),R_2(z)], \forall x,y,z \in A.
\end{align*}
\end{pro}

\begin{proof}
It follows immediately from Lemma \ref{commuting rota-baxter op} and Corollary \ref{CorLDendViaRB}.
\end{proof}

\begin{rmk}
Let $(A,[\c,\c])$ be a Lie-algebra. Recall that a trace function $\tau: A \to \mathbb{K}$  is a linear map such that $\tau([x,y])=0,\ \forall  x,y \in A$. When $\tau$ is a trace function, it is well known \cite{induced 3-Lie} that $(A,[\c,\c,\c]_{\tau})$ is a $3$-Lie algebra, where
$$[x,y,z]_{\tau}:=\circlearrowleft_{x,y,z \in A}\tau(x)[y,z],\  \forall x,y,z \in A.$$
Now let $\{R_1,R_2\}$ be a pair of of commuting Rota-Baxter operators (of weight zero) on the $3$-Lie algebra $(A, [\c,\c,\c]_{\tau})$. Then we can construct a $3$-L-dendriform structure on $A$, given by
\begin{align*}
 \nw(x,y,z)=& \tau (R_1R_2(x))[R_1R_2(y),z]+ \tau (R_1R_2(y))[z,R_1R_2(x)]+
      +\tau(z)[R_1R_2(x),R_1R_2(y)], \\
\ne(x,y,z)=& \tau (R_1(x))[R_1R_2(y),R_2(z)]+\tau(R_1R_2(y))[R_2(z),R_1(x)]
  + \tau(R_2(z)) [R_1(x),R_1R_2(y)],
\end{align*}
for any $x,y,z \in A$.

\end{rmk}
\section{Generalized derivations of $3$-L-dendriform algebras}
This section is devoted to investigate generalized derivations of   $3$-Lie   algebras, $3$-pre-Lie   algebras and $3$-L-dendriform algebras. Throughout the sequel $(A,<\cdot,\cdot,\cdot>)$ denotes a $3$-(pre)-Lie   algebra.

\begin{defi}
 A linear map $D: A \to A$ is said to be a derivation on $A$, if it satisfies the following condition (for $x,y ,z\in A$)
\begin{align}\label{TernaryDerivation}
& D(<x, y,z>)=<D(x) , y,z>+ <x,D(y),z>+<x,y,D(z)>.
\end{align}
\end{defi}

\begin{defi}
  Let $(A,\nw,\ne)$ be a $3$-L-dendriform algebra. A linear map $D: A \to A$ is said to be a derivation on $A$, if it satisfies the following condition ( for $x,y ,z\in A$)
\begin{align}\label{3DenAlgDerivation1}
& D(\nw(x, y,z))=\nw(D(x) , y,z)+ \nw(x,D(y),z)+\nw(x,y,D(z)),\\
\label{3DenAlgDerivation2}&D(\ne(x, y,z))=\ne(D(x) , y,z)+ \ne(x,D(y),z)+\ne(x,y,D(z)).
\end{align}
\end{defi}

We denote the set of all derivations on $A$ by  $Der(A)$ .
We can easily show that $Der(A)$ is equipped with a Lie algebra structure. In fact, for $D \in  Der(A)$  and $D' \in  Der(A)$, we have $[D,D'] \in Der(A)$, where $[D,D']$  is the standard commutator defined by $[D,D']= DD'-D'D$.
\begin{defi}
A linear mapping $D\in End(A)$ is said to be a quasi-derivation of $A$ if there exist linear mapping $D' \in End(A)$ such that
\begin{eqnarray}\label{TernaryQDerivation}
& D'(<x, y,z>)=<D(x) , y,z>+ <x,D(y),z>+<x,y,D(z)>,
\end{eqnarray}
for all $x,y,z \in A$. We say that $D$ associates with $D'$.
\end{defi}
\begin{defi}
  Let $(A,\nw,\ne)$ be a $3$-L-dendriform algebra. A linear map $D: A \to A$ is said to be a quasi-derivation on $A$, if it satisfies the following condition
\begin{align}\label{3DenAlgQDerivation1}
& D'(\nw(x, y,z))=\nw(D(x) , y,z)+ \nw(x,D(y),z)+\nw(x,y,D(z)),\\
\label{3DenAlgQDerivation2}&D'(\ne(x, y,z))=\ne(D(x) , y,z)+ \ne(x,D(y),z)+\ne(x,y,D(z)),
\end{align}
for all $x,y,z \in A$. We say that $D$ associates with $D'$.
\end{defi}
\begin{defi}
A linear mapping $D\in End(A)$ is said to be a generalized derivation  of $A$ if there exist linear mappings $D', D'',D''' \in End(A)$  such that
\begin{eqnarray}\label{TernaryGDerivation}
&&D'''(<x, y,z>)=<D(x) , y,z>+ <x,D'(y),z>+<x,y,D''(z)>,~~
\end{eqnarray}
for all $x,y,z \in A$. We also say that $D$ associates with $(D',D'',D''')$.
\end{defi}
\begin{defi}
  A linear map $D: A \to A$ is said to be a generalized derivation of a $3$-L-dendriform algebra  $(A,\nw,\ne)$, if it satisfies the following condition
\begin{align}\label{3DenAlgQDerivation1}
& D'''(\nw(x, y,z))=\nw(D(x) , y,z)+ \nw(x,D'(y),z)+\nw(x,y,D''(z)),\\
\label{3DenAlgQDerivation2}&D'''(\ne(x, y,z))=\ne(D(x) , y,z)+ \ne(x,D'(y),z)+\ne(x,y,D''(z)),
\end{align}
for all $x,y,z \in A$. We say that $D$ associates with $(D',D'',D''')$.
\end{defi}
The sets of  quasi-derivations and  generalized derivations will be denoted by  $QDer(A)$ and $GDer(A)$ , respectively.
It is easy to see that
$$Der(A)\subset QDer(A)\subset  GDer(A).$$
\begin{defi}
A linear map $\Theta \in End(A)$ is said to be a centroid of $A$ if
\begin{align}\label{TernaryCentroid}
& \Theta(<x, y,z>)=<\Theta(x), y,z>=<x,\Theta( y),z>=<x, y,\Theta(z)>,~~\forall~~ x,y,z \in A.
\end{align}
\end{defi}
\begin{defi}
  A linear map $\Theta: A \to A$ is said to be a Centroid of a $3$-L-dendriform algebra  $(A,\nw,\ne)$, if it satisfies the following conditions:
\begin{align}\label{3DenAlgCentroid1}
& \Theta(\nw(x, y,z))=\nw(\Theta(x) , y,z)= \nw(x,\Theta(y),z)=\nw(x,y,\Theta(z)),\\
\label{3DenAlgCentroid2}&\Theta(\ne(x, y,z))=\ne(\Theta(x) , y,z)= \ne(x,\Theta(y),z)=\ne(x,y,\Theta(z)),\ \forall \ x,y,z \in A.
\end{align}

\end{defi}
The  set of centroids of $A$ is denoted by $C(A)$.
\begin{defi}
 A linear map $\Theta \in End(A)$ is said to be a quasi-centroid of $A$ if
\begin{align}\label{TernaryQCentroid}
& <\Theta(x), y,z>=<x,\Theta( y),z>=<x, y,\Theta(z)>,~~\forall~~ x,y,z \in A.
\end{align}
\end{defi}
\begin{defi}
  A linear map $\Theta: A \to A$ is said to be a quasi-Centroid of a $3$-L-dendriform algebra  $(A,\nw,\ne)$, if it satisfies the following conditions:
\begin{align}\label{3DenAlgQCentroid1}
& \nw(\Theta(x) , y,z)= \nw(x,\Theta(y),z)=\nw(x,y,\Theta(z)),\\
\label{3DenAlgQCentroid2}&\ne(\Theta(x) , y,z)= \ne(x,\Theta(y),z)=\ne(x,y,\Theta(z)),\ \forall \ x,y,z \in A.
\end{align}

\end{defi}
The set of quasi-centroids of $A$ is denoted by $QC(A)$.
It is obvious  that $C(A)\subset QC(A)$.
\begin{pro}
Let $(A,\nw,\ne)$ be a $3$-L-dendriform algebra,  $D\in Der(A)$  and $\Theta\in C(A)$. Then
\begin{eqnarray*}
&& [D,\Theta]\in C(A).
\end{eqnarray*}
\end{pro}
\begin{proof} Assume that $D\in Der(A),~~\Theta \in C(A)$. For arbitrary $x,y \in A$, we have
\begin{align}
 D\Theta( \nw(x, y,z))&=D( \nw(\Theta (x), y,z))\nonumber
\\
 &=  \nw(D\Theta (x), y,z)+ \nw(\Theta (x),D( y),z)+ \nw(\Theta (x), y,D(z))\label{ternarycent1}
\end{align}
and
\begin{align}
 \Theta D( \nw(x, y,z))&=\Theta (\nw(D(x), y,z)+ \nw(\Theta (x),D( y),z)+ \nw(\Theta (x), y,D(z)))\nonumber
\\
 &=  \nw(\Theta D (x), y,z)+ \nw(\Theta (x),D( y),z)+ \nw(\Theta (x), y,D(z))\label{ternarycent2}.
\end{align}
By making the difference of  equations
\eqref{ternarycent1} and \eqref{ternarycent2}, we get
\begin{align*}
   & [D,\Theta] (\nw(x, y,z))=\nw([D,\Theta] (x), y,z).
\end{align*}
Similarly we can proof that, for all $x,y,z\in A,$ $$[D,\Theta] (\nw(x, y,z))=\nw(x,[D,\Theta] ( y),z),\  \ [D,\Theta] (\nw(x, y,z>))=\nw(x, y,[D,\Theta] (z))$$
and
 $$[D,\Theta] (\ne(x, y,z))=\ne([D,\Theta] (x), y,z)=\ne(x,[D,\Theta] ( y),z)=\ne(x, y,[D,\Theta] (z)).$$
\end{proof}
\begin{pro} $C(A)\subseteq QDer(A)$.
\end{pro}
\begin{proof}straightforward.
\end{proof}

\begin{pro}
Let $(A,\nw,\ne)$ be a $3$-L-Dendriform algebra and $D\in GDer(A)$ associates with $(D',D'',D''')$. Then $D$ is a  generalized derivation of associated horizontal $3$-pre-Lie algebras $(A,\{\cdot,\cdot,\cdot\}^h)$  and vertical $3$-pre-Lie algebras $(A,\{\cdot,\cdot,\cdot\}^v)$ defined in Proposition \ref{3LDendTo3PreLie} associates with $(D',D'',D''')$.
\end{pro}
\begin{proof}
Let $x,y,z\in A$ and  using definition of bracket $\{\cdot,\cdot,\cdot\}^h$ in \eqref{accolade horizintal}, we have
\begin{align*}
  D'''(\{x,y,z\}^h)=&D'''(\nw(x,y,z)+\ne(x,y,z)-\ne(y,x,z))\\
= &\nw(D(x),y,z)+\nw(x,D'(y),z)+\nw(x,y,D''(z))\\
&+\ne(D(x),y,z)+\ne(x,D'(y),z)+\ne(x,y,D''(z))\\
&-\ne(D(y),x,z)-\ne(y,D'(x),z)-\ne(y,x,D''(z))\\
=&\{D(y),x,z\}^h+\{y,D'(x),z\}^h+\{y,x,D''(z)\}^h.
\end{align*}
Similarly, we can proof that
$$ D'''(\{x,y,z\}^v)=\{D(y),x,z\}^v+\{y,D'(x),z\}^v+\{y,x,D''(z)\}^v.$$
\end{proof}
\begin{pro}
Let $(A,\{\cdot,\cdot,\cdot\})$ be a $3$-pre-Lie algebra and $D:A\to A$ be a generalized derivation on $A$ associates with $(D',D'',D''')$ and let
$R:A\to A$ be a Rota-Baxter operator of weight $0$
commuting  with $D,D',D"$ and $D'"$ . Then $D$ is a generalized derivation of the compatible $3$-L-dendriform algebra defined in Corollary \ref{CorLDendViaRB} associates with $(D',D'',D''')$.
\end{pro}
\begin{proof}Let $x,y,z\in A$. Using the definition of structures  $\nw,\ne$ given in \eqref{LDendViaRB}, we have
\begin{align*}
  D'''(\nw(x,y,z)) &= D'''(\{Rx,Ry,z\}) \\
   &=\{D(Rx),Ry,z\}+\{Rx,D'R(y),z\}+\{Rx,Ry,D''(z)\} \\
   &=\nw(D(x),y,z)+\nw(x,D'(y),z)+\nw(x,y,D''(z))
\end{align*}
and
\begin{align*}
  D'''(\ne(x,y,z)) &= D'''(\{x,Ry,Rz\}) \\
   &=\{D(x),Ry,Rz\}+\{x,D'R(y),Rz\}+\{x,Ry,D''(Rz)\} \\
   &=\ne(D(x),y,z)+\ne(x,D'(y),z)+\ne(x,y,D''(z)).
\end{align*}
\end{proof}

\begin{pro}
Let $(A,\nw,\ne)$ be a $3$-L-dendriform algebra. Then
\begin{itemize}
\item [a)]
$D(A)\oplus C(A) \subset QD(A)$.
\item [b)]
$QD(A)+QC(A) \subset GD(A)$.
\end{itemize}
\end{pro}
\begin{proof}
Straightforward.
\end{proof}


\end{document}